\documentclass[12pt]{article}

\usepackage{amsthm,amsmath,amsmath,amssymb,amsthm,amsfonts,amstext,amsbsy,amscd}
\RequirePackage[colorlinks=true,citecolor=blue,urlcolor=blue]{hyperref}
\usepackage[T1]{fontenc}
\usepackage[utf8]{inputenc}

\RequirePackage{hypernat}
\usepackage{graphicx}
\newenvironment{disarray}{\everymath{\displaystyle\everymath{}}\array} {\endarray}
\usepackage{pdfsync}
\usepackage{mathrsfs, amscd, amsfonts}

\newcommand{\field}[1]{\mathbb{#1}}
\newcommand{\R}{\field{R}}

\newcommand{\PP}{\field{P}}

\theoremstyle{plain}
\newtheorem{theorem}{Theorem} 
\newtheorem{definition}{Definition}

\newtheorem{lemma}{Lemma}
\newtheorem{prop}{Proposition}
\newtheorem{assumption}{Assumption}{\bf}{\rm}
\theoremstyle{remark}

{\smallskip}%

\newcommand{\E}{\field{E}}
\renewcommand{\P}{\field{P}}

\newcommand{\pl}{\mathcal{P}_{\lambda}}

\newcommand{\nul}{\nu_{\lambda}}
\newcommand{\vl}{\varphi_{\lambda}}

\newcommand\nbOne{{\mathchoice {\rm 1\mskip-4mu l} {\rm 1\mskip-4mu l}
{\rm 1\mskip-4.5mu l} {\rm 1\mskip-5mu l}}}

\begin{document}



\title{Statistical estimation of  jump rates for a specific class of Piecewise Deterministic Markov Processes.}

\author{ N. Krell\thanks{Universit\'e de Rennes 1,
 Institut de Recherche math\'ematique de Rennes,
CNRS-UMR 6625, Campus de Beaulieu.
 B\^atiment 22, 35042 Rennes Cedex, France. {\tt email}: nathalie.krell@univ-rennes1.fr}
}

\maketitle
\begin{abstract}
 We consider the class of   Piecewise Deterministic Markov Processes (PDMP), whose state space is $\R_{+}^{*}$, that possess an increasing deterministic motion and that shrink deterministically when they jump. Well known examples for this class of processes are Transmission Control Protocol (TCP) window
size process   and the processes modeling the size of a "marked" {\it Escherichia coli} cell. Having observed the PDMP until its $n$th jump, we construct a nonparametric estimator of the jump rate $\lambda$. Our main result is that for  $D$ a compact subset of $\R_{+}^{*}$, if  $\lambda$  is in the H{\"{o}}lder space ${\mathcal H}^s({\mathcal D})$, the squared-loss error of the estimator is asymptotically close to the rate of $n^{-s/(2s+1)}$. Simulations    illustrate
the behavior of our estimator.
\end{abstract}

{\small 
}

\section{Introduction}

The Piecewise deterministic Markov processes were first introduced in the literature by Davis (\cite{Dav84} and
\cite{Dav93}), they
 form a family of càdlàg Markov processes involving a deterministic
motion punctuated by random jumps. We refer to the paper \cite{kre} and its references for an overview of PDMPs. Let us detail the special case of PDMPs that
will be considered in this paper. The motion of the PDMP $(X(t))_{t\geq 0}$
depends on three local characteristics, namely the jump rate $\lambda$, the flow $\phi$
 and a deterministic increasing function $f$ which governs the location of the process at the jump time (in the general case it depends on a Markov kernel $Q$). The process starts from $x$ and follows the flow $\phi(x, t)$
until the first jump time $T_{1}$
 which occurs spontaneously in a Poisson-like fashion with rate $\lambda(\phi (x,t))$. The location of the process at the jump time $T_{1}$, denoted by $Z_1=X(T_{1})$, is equal to $f(X(T_{1}^-))$
, with $f$ a function such that  $0\leq f(y)\leq\kappa y$ with $0<\kappa< 1$. The motion restarts from this new point $Z_1$ as before. This fully describes a piecewise continuous trajectory for $\{X(t)\}$
with jump times $\{T_{k}\}$
and post jump locations $\{Z_{k}\}$, and which evolves according to the flow $\phi$
 between two jumps.

This paper analyzes a special case of both  Piecewise Deterministic Markov process (PDMPs) and growth-fragmentation model.  

The fisrt  known example for this class of processes is the  TCP window
size process (see  \cite{BaCh}, \cite{ChMaPa}, \cite{DGR02}, \cite{GrKa} and
 \cite{GRZ04}). 
The TCP protocol is one of the main data transmission protocols of the Internet. It has been designed to adapt to the various traffic conditions of
the actual network. For a connection, the maximum number of packets that can be sent at each round is given by a variable
$W$, called the
congestion window size. If all the
$W$
packets are successfully
transmitted, then
$W$
is increased by 1, otherwise it is multiplied by
$\delta\in(0,1)$ (detection of
a congestion). As shown in \cite{DGR02} a correct scaling of this process leads to a continuous
time Markov process, called general TCP window size process.

The second example is  the processes modeling the size of a "marked" {\it Escherichia coli} cell   (see \cite{hof} and \cite{LaPe}). If by following the evolution of {\it Escherichia coli} bacterium, we chose at random a bacteria, and follow his growth, and at his division, we choose randomly,  and independently of the process, to follow one of his    daughter  and on so one. We call the bacteria follow at each time, the "marked" bacteria. The size of the "marked" bacteria is a PDMP. More precisely between the jumps, the bacteria grows exponentially with a growth rate which we will refer to later as the  instantaneous growth rate. 
The division rate of such process is denoted by $B(\cdot)$
 and at each jump its size is divided by two. 
 
  In both cases the value of the process is divided in a deterministically way at each jump.

  The purpose of this article is to perform non-parametric estimation of the jump rate based on a single observation  of the process over a long time interval. 
  
  The assumptions made in this paper ensure that the Markov process which gives the size after each jump of the PDMP is ergodic. The ergodic theorem  was already known for some  cases of PDMPs  and for some one-dimensional jump-diffusions.  In  \cite{Cloez}  Cloez studies a process which is more general than the one  considered in this paper. More precisely, his process evolves like a diffusion which satisfies a stochastic differential equation between the jumps, but he requires that $\lambda$ be bounded  below. This paper does not make  such assumption. In addition, the well known ergodicity properties of  TCP, due to J.-M. Bardet et al. \cite{BaCh}, provide quantitative estimates for the exponential convergence to equilibrium, in terms of the total variation and Wasserstein distances. However, such results cannot be used in the present framework as we need a uniform upper bound for the speed of convergence to  the invariant measure  over a certain class of functions in order to prove the statistical result.

Our approach is based on the methods used by Doumic et al.
\cite{hof}, which were
applied to analyse a  special case of PDMPs dealing with  “marked” bacteria size evolution.  The analytical results of this paper can be generalized to more general PDMPs, for example to TCP.  In \cite{hof} the authors
 do not merely study the "marked" bacterium which is selected at each division uniformly and independently, but  the evolution of all bacteria involved. A dependence structure results when, for example, a bacteria divides and gives birth to two new bacterium of equal size. In a general case, the instantaneous growth rate (constant for a given bacteria) would depend on the bacteria itself. Consequently, the size of the bacterium is no longer a Markov process. However, on the other hand,  the size of the process and the instantaneous growth rate, together form  a Markov process. The present paper contains the case of a "marked" bacterium, in its simple case, where the  instantaneous growth rate is the same for all bacteria.

This article has two main features that may be useful for future studies. First, it can be used as a tool to verify the PDMP jump rate proposed in the existing literature, for example $\lambda (x)=x$ in the TCP case. Secondly, it  can suggest an estimator for some special  cases, such as for {\it
 Escherichia coli} bacteria, where the jump rate is not known.

In  \cite{Azaduf2} Aza{\"{i}}s  et al. give an estimator of the conditional distribution of the inter-jump times for a PDMP, which is uniformly consistent when only one observation of the process within a long time is available. They deal with PDMPs which jump when they hit the boundary (this case is not considered in our paper). Their method relies on a generalization of Aalen’s multiplicative intensity model \cite{AalPhd,Aal77, Aal78}. But they only prove the uniform consistency of their estimator. They also have to  assume that the process $(X(t))_{t\geq 0}$
evolves in a bounded space. Here we do not make this assumption. As a consequence the tools of their paper and of the present one are different. To the best of my knowledge, \cite{Azaduf2} is the only work investigating the nonparametric estimation of the conditional distribution of the inter-arrival times for PDMPs. This paper relies on \cite{Az12b} in which the authors focus on the non parametric estimation of the jump rate and the cumulative rate for a class of non homogeneous marked renewal processes. The case where the post-jump locations of the PDMP do not depend on inter-arrival times was considered in \cite{Az12b}.

We refer to  \cite{Azaduf2} for an overview of the statistical methods related to this kind of process, as well as to \cite{Aal77,Aal78,AalPhd}
 for statistical inference related to the multiplicative intensity model. The book of Andersen et al. \cite{andbog} gives a comprehensive account of estimation for jump rates which depend both on time and spatial variable.

As far as I know, the only other paper dealing with general PDMP is the work of Aza{\"{i}}s \cite{Az}, where the author focuses on a non parametric recursive estimator for the transition kernel of the PDMP.

Other works dealing with  specific cases of PDMP can been seen as ruin probability, for example, as found in the references of \cite{asm}. In addition, the PDMP modeling the quantity of a given food contaminant in the body has been studied in  \cite{bou,ber1, ber2}, assuming that the inter-intake times
are i.i.d.. In this paper we do not make this assumption.

The paper of Aza{\"{\i}}s  and Genadot \cite{AzGe}
 consider a growth-fragmentation model where $\lambda$ is constant and the Markov kernel $Q$ is absolutely continuous with respect to the Lebesgue measure. The case considered is totally different from the present one.

The paper is organized as follows. In section \ref{2}
 we introduce the class of PDMPs, that will be studied, and we give an explicit construction of the PDMP. Section \ref{3}
concerns the statistical estimation of the jump rate. We first define the observation scheme and the class of functions for the 3 parameters of the PDMP concerned, and the assumptions used (in Subsection \ref{upper bound}). In subsection \ref{2.2}, some ergodicity results are stated uniformly over the class of functions previously  defined. We explicitly construct an estimator  $\widehat \lambda_n$ of $\lambda$. In subsection \ref{3.4}
 an upper bound for the squared-error loss is given in the main Theorem \ref{upper bound2}. In \ref{simulation}, we illustrate our result with simulations of a TCP process which could not been seen as a "marked" bacteria process, amongst others. 

Finally, section \ref{4} is reserved for the proofs. In subsection \ref{4.1} we prove  the ergodicity result of Subsection \ref{2.2}. Subsection \ref{rate of convergence for the empirical measure}
 presents the intermediate results needed in  \ref{4.3}
 to prove the major results given in Subsection \ref{3.4}.

\section{PDMP}\label{2}

In general a PDMP is defined by its local characteristics, namely, the jump rate $\lambda$, the flow $\phi$ 
 and the transition measure $Q$ according to which the location of the process is chosen at the jump time. In this article, we consider  a specific class of PDMP which includes the control of congestion TCP/IP used in communication networks (V. Dumas and al \cite{DGR02}, V. Guillemin and al. \cite{GRZ04}), for which the transition measure $Q$ is a Dirac mass function, which means that when the process jumps, the size after the jump is a deterministic function of its size before. More precisely, 

 \begin{assumption}\label{hypopdmp}
\begin{itemize}

\item The flow $\phi:\R^+ \times\R^+ \rightarrow\R^+$ is a one-parameter group of homeomorphisms:
$\phi$ is $\mathcal{C}^{1}$, $\phi (., t)$ is an homeomorphism for each $t\in\R^+$, satisfying the semigroup property: $\phi(., t+
s) = \phi (\phi(., s), t) $ and $\phi_x (.):=\phi (x, .)$ is an $\mathcal{C}^{1}$-diffeormorphism.

\item The jump rate $\lambda :\R^+ \rightarrow \R_{+}$ is assumed to be a measurable function
satisfying
\[\forall x\in \R^+ ,\;\; \exists\:\varepsilon>0\:\; \text{such that}\:\; \int_{0}^{\epsilon}\lambda(\phi (x,s))ds<\infty .\]

\item $f:\R^+\rightarrow \R^+$ is an increasing $\mathcal{C}^{1}$-diffeomorphism  and $Q(u,\{y\})=\nbOne_{\{f(u)=y\}}$. 

\end{itemize}
\end{assumption}

Given these three characteristics, it can be shown (\cite{Dav93}, pages 62-66), that there exists
a filtered probability space  $(\Omega , \mathcal{F},\{\mathcal{F}_{t}\}, \{\mathbb{P}_{x}\})$ such that the motion of the
process $\{X(t)\}$ starting from a point $x\in \R^+$ may be constructed as follows.
Consider a random variable $T_{1}$ such that
\begin{equation}\mathbb{P}_{x}(T_{1}>t)=
   e^{-\Lambda (x,t)}  ,
\label{t1}\end{equation}
where for $x\in \R^+$ and $t \in \R^+$
\[
\Lambda (x,t)=\int_{0}^{t}\lambda (\phi (x,s))ds.
\]
 If $T_{1}$ is equal to infinity, then the process $X$ follows the flow, i.e. for
$t\in \R_{+}$, $X(t) = \phi (x, t)$. Otherwise let \begin{equation}Z_{1}=f(\phi (x,T_{1})).\label{Z1}\end{equation} The trajectory of $\{X(t)\}$ starting at
$x$, for $t \in [0,T_{1}]$, is given by
\[
X(t)=
\begin{cases}
\phi (x,t) &\text{for } t<T_{1},\\
Z_{1}      &\text{for } t= T_{1}.
\end{cases}
\]
Inductively starting from $X(T_{n}) = Z_{n}$, we now select the next inter-jump time $T_{n+1}- T_{n}$
and post-jump location $X(T_{n+1}) = Z_{n+1}$ in a similar way.

This construction properly defines a strong Markov process $\{X(t)\}$   with jump times $\{T_{k}\}_{k\in\mathbb{N}}$
(where $T_{0} = 0$).  A very natural Markov chain is linked to $\{X(t)\}$, namely the jump chain
$(Z_{n})_{n\in\mathbb{N}}$.

$\{X(t)\}$ is a Markov process with infinitesimal generator $G$:
\begin{equation}\label{generator}
Gh(y)=\phi_{x}^{'} (y)h^{'}(y)+\lambda (y)\big(h(f(y))-h(y)\big) 
\end{equation}
for $h: \R^+\rightarrow \R$ a bounded measurable functional.

Thanks to (\ref{t1}), we get that
\[\P (T_{1}\in dt | Z_{0}=x)=\lambda (\phi_{x }(t))e^{-\int_{0}^{t}\lambda (\phi_{x}(s))ds}dt.\]
Using (\ref{Z1}), the monotonicity of $f\circ \phi_x$ and  a simple change of variables, we get the transition probability  of the Markov chain $(Z_{n})_{n\in\mathbb{N}}$:

\begin{equation}\P (Z_{1}\in dy | Z_{0}=x)=\lambda (f^{-1}(y))e^{-\int_{f(x)}^{y}\lambda (f^{-1}(s))g_{x}(s)ds}g_{x}(y)\nbOne_{\{y\geq f(x)\}}dy,\label{densitet1}\end{equation}
where \[g_{x}(y)=\left[(f\circ \phi_{x})^{'}\left((f\circ \phi_{x})^{-1}(y)\right)\right]^{-1}\] and $\phi_x (.):=\phi (x,.)$.

\section{Statistical estimation of the jump rate}\label{3}

\subsection{The observation scheme}\label{3.1}
Statistical inference is based on the observation scheme:
\[(X(t), t\leq T_n)\]
and asymptotics are considered when the number of jumps  of the process, $n$, goes to infinity.

Actually the  simpler observation scheme:
\[(X(T_i ), 1\leq i\leq n)=(Z_i , 1\leq i\leq n)\] is sufficient.

\subsection{Class of functions} \label{upper bound}
We want to bound from above the squared-loss error of our estimator over compact intervals ${\mathcal D}$ of $\R_{+}^{*}$. We need to specify the local smoothness properties of $\lambda$ over ${\mathcal D}$, together with general properties that ensure that  the empirical measurements of the PDMP converge toward  the invariant probability with an appropriate speed of convergence. So we have to impose technical assumptions on $\lambda$ in particular near the origin and infinity.

\begin{definition}  For $b>0$, a vector of positive constants 
$\mathfrak{c}=(r, \kappa, l, L, a)$, and two positive functions $m:[0,\infty)\rightarrow (0,\infty)$ and $M:[0,\infty)\rightarrow (0,\infty)$ such that for all $x\in[0,\infty):$ $M(x)\geq m(x)>0$, we
introduce the class ${\mathcal F}(\mathfrak{c}, m, M)$  
of triples of continuous functions
$\lambda:[0,\infty)\rightarrow [0,\infty)$, $f:\R^+\rightarrow \R^+$  and  $\phi:\R^+ \times\R^+ \rightarrow\R^+$ such that 
\begin{equation}\int_{0}^{f(r)}M(s)\lambda (f^{-1}(s))ds \leq L,\label{lambdaL}\end{equation}

\begin{equation}\int_{f(r)}^{\infty}m(s)\lambda (f^{-1}(s))ds =\infty ,\label{lambdainf}\end{equation}

\begin{equation}\int_{f(r)}^{r}M(s)\lambda (f^{-1}(s))ds \geq l,\label{lambdal}\end{equation}

\begin{equation}\lambda (x)\geq \frac{a (f(x))^{b}}{m(f(x))}\:\; \forall x\geq r,\label{lambdapoly}\end{equation}

 \begin{equation}\forall x>0, \:\;\: 0<f(x)\leq \kappa x, \label{fk} \end{equation}
\begin{equation} \forall y>0, \:\;\:\forall x\geq 0, \:\;\:  m(y)\leq g_{x}(y)\leq M(y),\label{gmM}\end{equation} where \[g_{x}(y)=\left[(f\circ \phi_{x})^{'}\left((f\circ \phi_{x})^{-1}(y)\right)\right]^{-1}\] and $\phi_{x}(\cdot) =\phi (x,\cdot)$.

\end{definition}

 We notice that $f(0)=0$.

Typically an interesting case would be $f(x)=\kappa x$ with $\kappa\in (0 ,1)$ and then \eqref{gmM} would simply be 
\[M^{-1}(y)\leq \kappa \phi_{x}^{'}( \phi_{x}^{-1}(y/\kappa))\leq m^{-1}(y).\] This seems  quite 	reasonable because, in view of the definition of the infinitesimal generator of the PDMP defined in \eqref{generator},  we would like $\phi_{x}^{'} (\cdot )$ not to be identically zero. 

Also the cases where $\phi_x (t)= xe^{\alpha t}$ or $\phi_x (t)= x+t$ satisfy \eqref{gmM} over ${\mathcal D}$.

 Define
\[
\delta(\mathfrak{c},f):= \frac{1}{1-\kappa^{b+1}} \exp\big(- (1-\kappa^{b+1})\tfrac{am}{b+1} (f(r))^{b+1}\big).
\]

The last assumption that we will need is:
\begin{assumption}\label{contrainte constante}
\begin{equation}\exists b>0  :(\kappa^{b +1}-1)\frac{am}{b +1}(f(r))^{b +1}<\log (1-\kappa^{b+1}),\label{lambdaexis}\end{equation}
\end{assumption}
 so that we have $\delta(\mathfrak{c},f) < 1$

Fix a vector of positive constants $\mathfrak{c}=(r, \kappa, l, L, a)$, a constant $b$ and a function $f$.

\subsection{Geometric ergodicity of the discrete model}\label{2.2}

Let $x\in \R^+$. Introduce the transition kernel 
\[
\pl(x,dx') = \PP\big(Z_n\in dx'\big|\,Z_{n-1}=x\big)
\]
of the size of the process at the $n$th jump time, given the size of the process at the $(n-1)$th jump time.
From \eqref{densitet1}, we infer that $\P (Z_{1}\in dy | Z_{0}=x)$ is equal to
\[
\lambda (f^{-1}(y))e^{-\int_{f(x)}^{y}\lambda (f^{-1}(s))g_{x}(s)ds}g_{x}(y)\nbOne_{\{y\geq f(x)\}}dy.
\]

Thus   we obtain an explicit formula for
\[\pl(x,dy)=\pl (x,y) dy\] with
\begin{equation} 
\mathcal{P}_{\lambda}(x,y)=\lambda (f^{-1}(y))e^{-\int_{f(x)}^{y}\lambda (f^{-1}(s))g_{x}(s)ds}g_{x}(y)\nbOne_{\{y\geq f(x)\}}.  \label{density explicit}
\end{equation}
Denote the left action of positive measures $\mu(dx)$ on $\R^+$ for the transition kernel $\pl$ by
\[\mu\pl(dy)=\int_{\R^+}\mu(dx)\pl(x,dy)\]
 and the right action of a function $\psi$ on $\R$ for the transition $\pl$ by
\[\pl \psi(x)=\int_{\R^+}\psi (y)\pl(x,dy)\]
We now give the geometric ergodic theorem that we will need for the statistical part. We need an uniformity 
 on the class of functions  ${\mathcal F}(\mathfrak{c},m,M)$   defined in subsection \ref{upper bound}.

We introduce the Lyapunov function
\begin{equation} \label{first def Lyapu}
{\mathbb V}(x)=\exp\left(\frac{a}{b+1}(f(x))^{b+1}\right)\;\;\text{for}\;\;x\in \R^+.
\end{equation}
The function ${\mathbb V}$ controls the rate of the geometric ergodicity of the chain with transition $\pl$ and appears in the proof of Proposition \ref{prop transition}.

\begin{prop}\label{prop transition}
Under Assumption \ref{hypopdmp},  for every $ \lambda$ such that $(\lambda,f, \phi )\in {\mathcal F}(\mathfrak{c},m,M)$ there exists a unique invariant probability measure of the form $\nul (dx)=\nul (x)dx$ on $\R^+$. Moreover there exist $0<\gamma <1$, a constant $R$ and a function $\mathbb{V}:\R^+\rightarrow [1,\infty )$ such that 

\begin{equation} \label{contraction}
\sup_{\lambda \in {\mathcal F}(\mathfrak{c},m,M)}\sup_{|\psi| \leq V}\big|\pl^k \psi(x)-\int_{\R^+}\psi(z)\nul (z)dz\big| \leq R{\mathbb V}(x)\gamma^k
\end{equation}
for every $x\in \R^+$, $k\geq 0$,  where the supremum is taken over all functions $\psi:\R^+ \rightarrow \R$ satisfying $|\psi (x)| \leq {\mathbb V}(x)$ for all $x\in \R^+$. 
The function ${\mathbb V}$ is $\nul$-integrable for every $\lambda$ such that $(\lambda,f, \phi )\in {\mathcal F}(\mathfrak{c},m,M)$.

For all $y\in \R^+$ we have the relation:
\begin{equation}\lambda (y)\E_{\nul}(g_{Z_0}(f(y)) \nbOne_{\{f(Z_0 )\leq f(y)\}} \nbOne_{\{Z_1 \geq f(y)\}})=\nul (f(y)).\label{formulel}\end{equation}
\end{prop}

\subsection{Construction of a nonparametric estimator}\label{3.2}
By formula \eqref{formulel}, 
\[
\lambda (y)=\frac{\nul (f(y))}{\E_{\nul}(g_{Z_0}(f(y)) \nbOne_{\{f(Z_0 )\leq f(y)\}} \nbOne_{\{Z_1 \geq f(y)\}})},
\]
provided the denominator is positive. This representation  suggests an estimation procedure, replacing the marginal density $\nul(f(y))$ and the expectation in the denominator by their empirical counterparts. To that end,
pick a kernel  function
$$K: \R\rightarrow [0,\infty),\;\;\int_{\R}K(y)dy=1,$$ 
and set $K_h(y)=h^{-1}K\big(h^{-1}y\big)$ for $y\in \R$ and $h>0$. Our estimator is defined by
\begin{align} 
\widehat \lambda_n (y) = \frac{n^{-1}\sum_{k=1}^{n}  K_{h_{n}}(Z_{k}-f(y))}{n^{-1}\sum_{k=1}^{n}  g_{Z_{k-1}}(f(y))\nbOne_{\{Z_{k}\geq f(y),\,f(y) \geq f(Z_{k-1})\}}\vee \varpi_n }.
\label{def estimator}
\end{align}
where $\varpi_n >0$ is a threshold that ensures that the estimator is well defined in all cases and $x \vee y = \max\{x,y\}$. Thus $(\widehat \lambda_n(y), y\in {\mathcal D})$ is specified by the choice of the kernel  $K$, the bandwidth $h_n >0$ and the threshold $\varpi_n >0$. 
\begin{assumption} \label{prop K} The function $K$ has compact support, and for some integer $n_0 \geq 1$, we have
$\int_{\R}x^kK(x)dx={\bf 1}_{\{k=0\}}\;\;\text{for}\;\;0 \leq k\leq n_0.$
\end{assumption}

\subsection{Rate of convergence}\label{3.4}
We are now ready to state our main result. For $s>0$, with $s=\lfloor s\rfloor + \{s\}$, $0< \{s\} \leq 1$ and $\lfloor s\rfloor$ an integer, introduce the H\"older space ${\mathcal H}^s({\mathcal D})$ of functions $f:{\mathcal D}\rightarrow \R$ possessing a derivative of order $\lfloor s \rfloor$ that satisfies
\begin{equation} \label{def sob}
|f^{\lfloor s \rfloor}(y)-f^{\lfloor s \rfloor}(x)| \leq c(f)|x-y|^{\{s\}}.
\end{equation}
The minimal constant $c(f)$ such that \eqref{def sob} holds defines a semi-norm $|f|_{{\mathcal H}^s(\mathcal{D})}$. We equip the space ${\mathcal H}^s(\mathcal D)$ with the norm 
\[\|f\|_{{\mathcal H}^s(\mathcal D)} = \|f\|_{L^\infty({\mathcal D})} + |f|_{{\mathcal H}^s({\mathcal D})}\] and the associated H\"older balls
\[{\mathcal H}^s({\mathcal D}, M_1 ) = \{\lambda :\;\|\lambda\|_{{\mathcal H}^s({\mathcal D})} \leq M_1 \},\;M_1 >0.\]
\begin{theorem} \label{upper bound2}
Work under Assumption  \ref{hypopdmp} and Assumption \ref{contrainte constante}. 
Specify $\widehat \lambda_n $ with a kernel $K$ satisfying Assumption \ref{prop K} for some $n_0>0$ and
\[h_n =c_0n^{-1/(2s+1)},\;\;\varpi_n\,\:\; \text{such that} \:\; \lim_{n\rightarrow \infty} \varpi_n =0\]
For every $M_1 >0$ and  $M_2 >0$,  there exist $c_0=c_0(\mathfrak{c}, M_1 , M_2)$ and $d(\mathfrak{c})\geq0$ such that for every $0<s<n_0$ and every compact interval ${\mathcal D}\subset (d(\mathfrak{c}),\infty)$ such that $\inf {\mathcal D} \geq f(r)$, we have
\[
\sup_{\lambda}\E_{\mu}\big[\|\widehat \lambda_n-\lambda\|_{L^2({\mathcal D})}^2\big]^{1/2} \lesssim \varpi_n^{-2} n^{-s/(2s+1)},\]
where the supremum is taken over
\[ (\lambda,f, \phi )\in {\mathcal F}(\mathfrak{c},m,M), \:\; \:\; \lambda \in  {\mathcal H}^s({\mathcal D}, M_{1}),\:\; \:\; \:\;g_x \in {\mathcal H}^s({\mathcal D}) ,\:\; \:\;\|f\|_{L^\infty({\mathcal D})} \leq M_2 ,\:\; \:\;\text{and}\:\; f^{-1} \in{\mathcal H}^s({\mathcal D})\]
and $\E_{\mu}[\cdot]$ denotes expectation with respect to any initial distribution $\mu(d\boldsymbol{x})$ for $(Z_{0})$ on $\R^+$ such that $\int{\mathbb V}(x)^2\mu(dx)<\infty$.
 \end{theorem}

We observe that we recover the result for the marked bacteria of \cite{hof}. In this case $\phi (x,t)=xe^{\kappa_0 t}$ with $\kappa_0\in\R^{+ *}$ and $f(x)=x/2$, so that $g_x (y)=\frac{1}{\kappa_0 y}$. We find the same estimator but the speed of convergence is a little bit better, as $\varpi_n$  need not be $\log (n)$; rather we only require that $\lim_{n\rightarrow \infty} \varpi_n^{-1} =0$.

\subsection{Numerical implementation}\label{simulation}
The goal of this subsection is to illustrate the asymptotic behaviour of our estimator via numerical experiments. More precisely we first investigate numerical simulations for the TCP.

The TCP window-size process appears as the scaling limit of the transmission rate of a server uploading packets
on the Internet according to the algorithm used in the TCP (Transmission Control Protocol) in order to avoid
congestion (see \cite{DGR02} for details on this scaling limit). This PDMP takes  values in $\mathbb{R}^+$ and the jump rate $\lambda$ is the identity function. The function $f$ which represents the proportion of the size kept after the jump is $f(x)=x/2$. The flow is $\phi(x,t)=x+t$.

As a consequence  the size of the process after the $n$-th jump $Z_n$, conditional on $Z_{n-1}$, has the same law as $\sqrt{(1/4)Z_{n-1}^2+e_n/2}$, where  $(e_n )_{n\geq 0}$ is a family of i.i.d. random variables with exponential distribution of parameter 1. The variable $e_n$ is also independent of $(Z_i )_{i\leq n-1}$. 
As a consequence it is  easy to generate the $(Z_n )_{n\geq 0}$ recursively. A trajectory of such a PDMP is given in Figure \ref{fig1}.  These processes satisfy the assumptions required for our Theorem, with $\kappa=1/2$, $m(\cdot )=1$, $M(\cdot ) =1.1$, $g_\cdot (\cdot)=1/2$, $r=1$, $a=1$, $b=1/2$, $\varpi_n= (\log (n))^{-1}$, $h_n=n^{-1/3}$ and $K(x)=(2\pi )^{1/2} exp( -x^2/2)$. With the Gaussian Kernel, for which $n_0=1$ for Assumption 2, we expect a rate of convergence of order $n^{1/3}$
at best.

We display our numerical results as specified above in Figures \ref{fig1}, \ref{fig2} and \ref{fig3}.

Figure \ref{fig2} displays the reconstruction of $\lambda$  for different simulated samples, for $n=1000$, $n=10000$ and $n=100000$. As expected, the estimation is better for larger $n$. The estimator performs worse fo small $x$ as these sizes are rarely reached by the TCP process.

In figure \ref{fig3}, we plot  the empirical mean error of our estimation procedure on a $\log$-$\log$ scale. The numerical results agree with the theory.

\begin{figure}[h]
   \begin{center}
      \includegraphics[height=8.5cm]{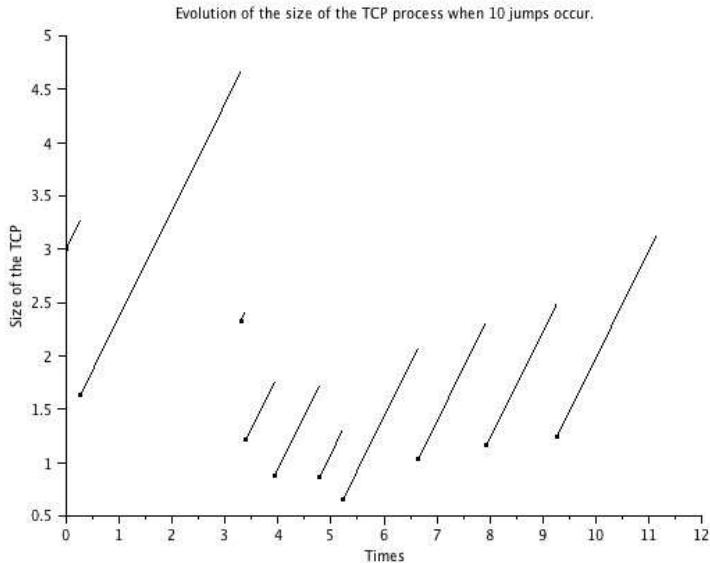}
   \end{center}
   \caption{\label{fig1}\footnotesize Evolution of the TCP process when 10 jumps occur.}
\end{figure}
\begin{figure}[h]
\includegraphics[height=4.125cm]{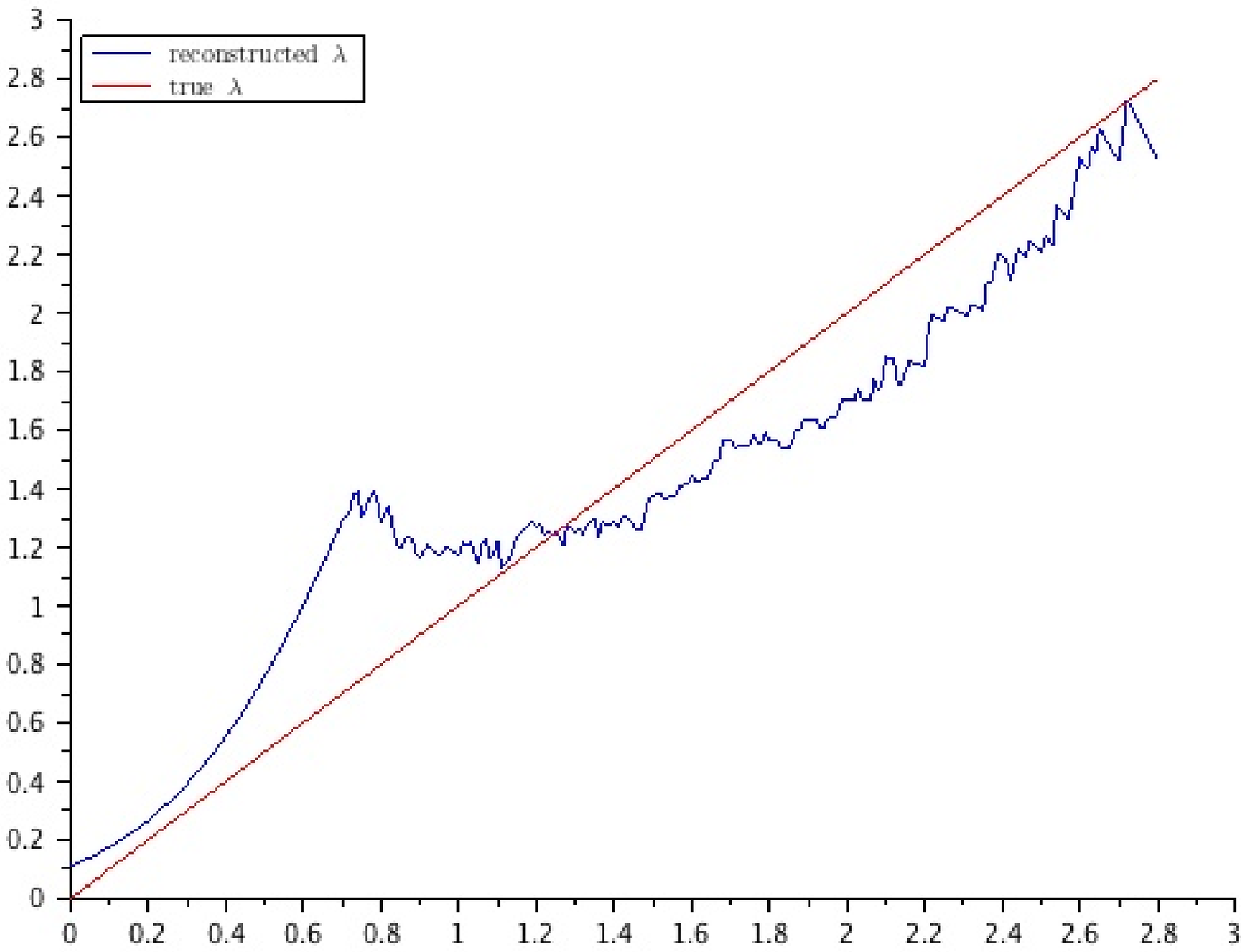}\includegraphics[height=4.15cm]{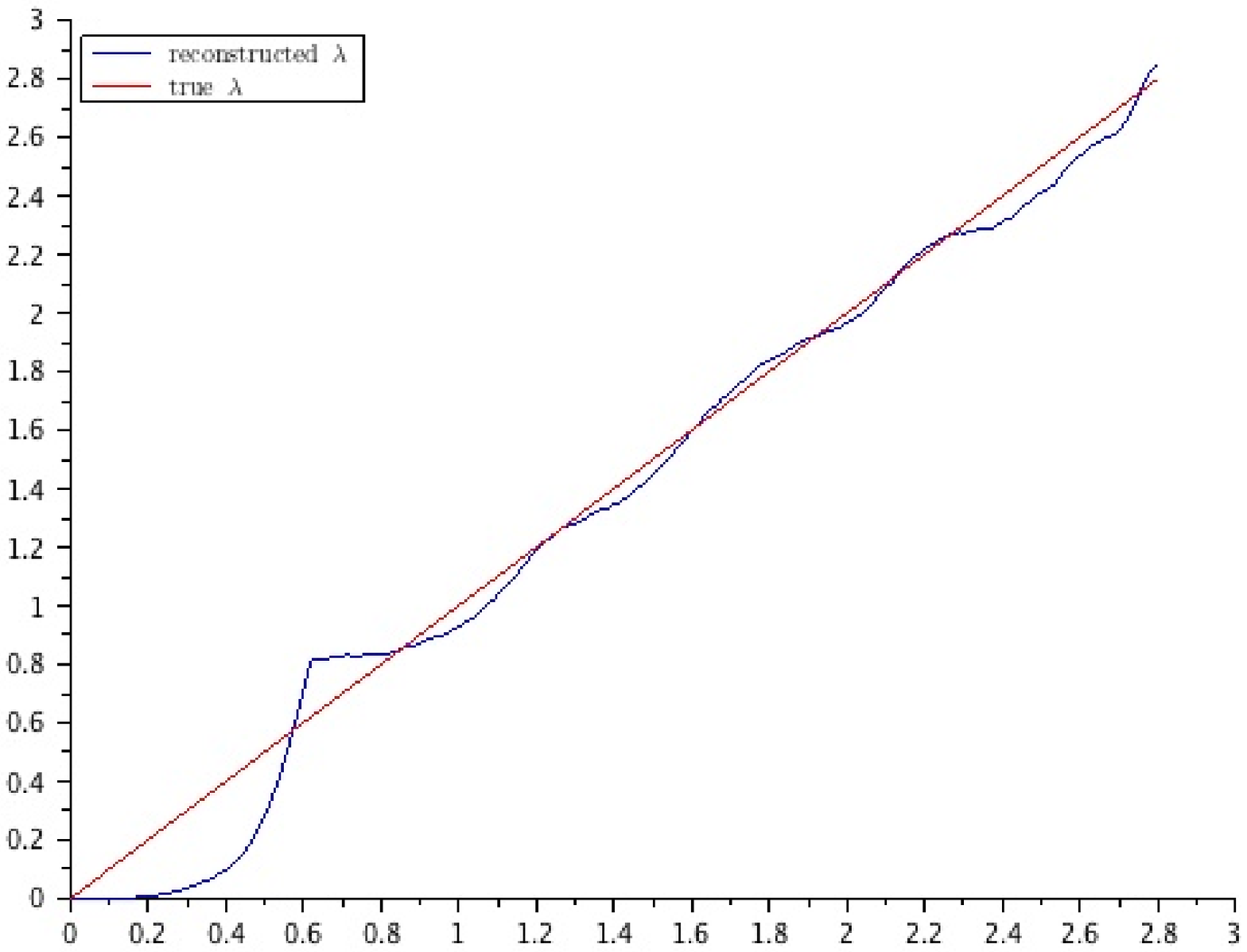}\includegraphics[height=4.125cm]{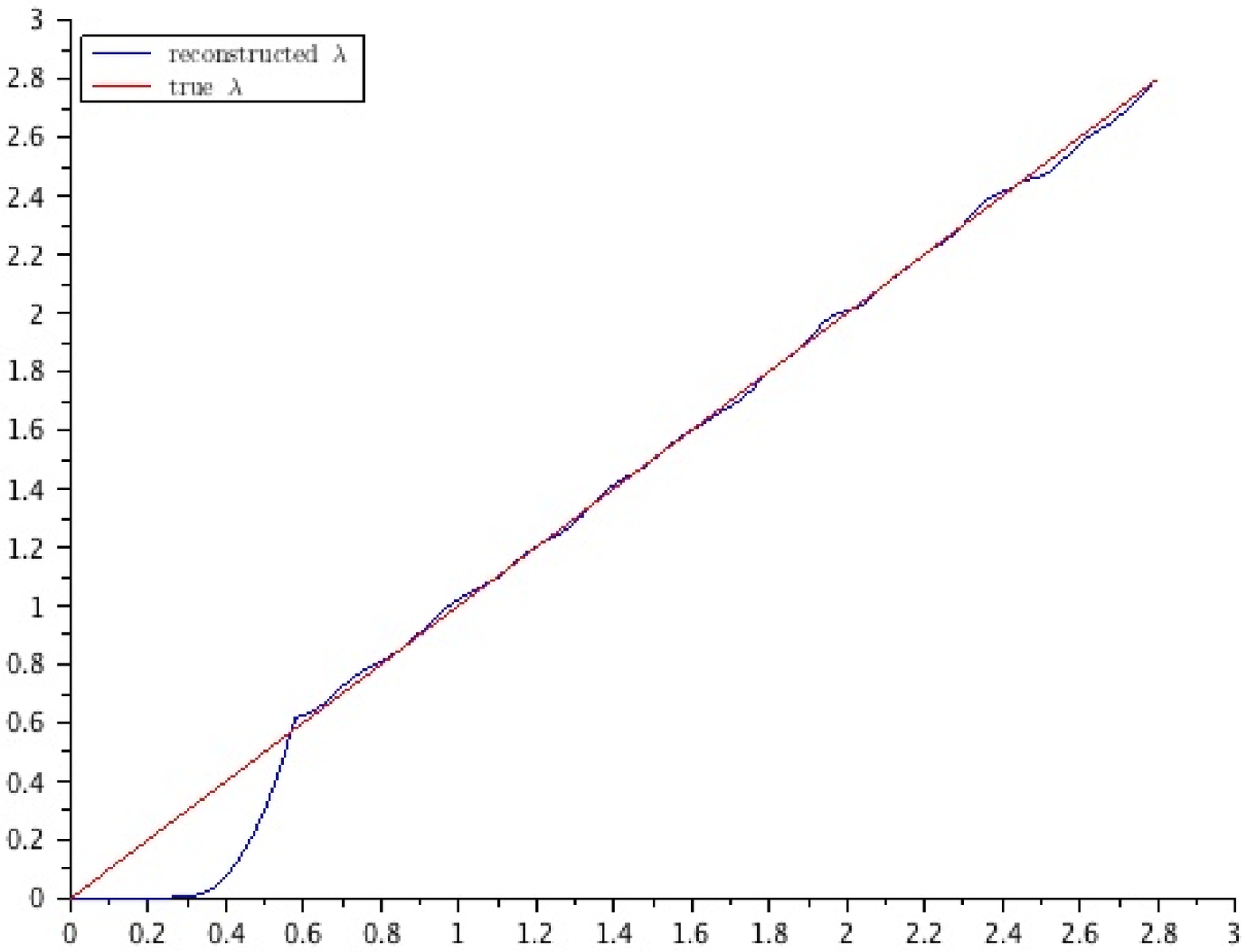}
 \caption{\label{fig2}\footnotesize Reconstruction of $\lambda$ for $n=1000$, $n=10000$ and $n=100000$ in the TCP case.}
\end{figure}
\begin{figure}[h]
   \begin{center}
 \includegraphics[height=7cm]{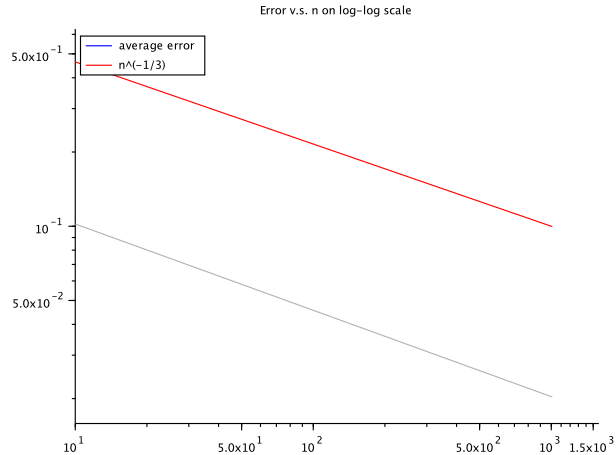}
    \end{center} \caption{\label{fig3}\footnotesize The empirical mean error of the estimation procedure v.s.  the theoretical rate on a log-log scale in the TCP case.}
\end{figure}

We now consider  a bacteria marked case, for which the size of the bacterium are divided by 3 instead of 2 at division. Thus we have $f(x)=x/3$, $\phi (x,t)=xe^t$ and therefore $g_x (y)=1/y$. Let $(e_{i})_{i\in\mathbb{N}^*}$ a i.i.d family of exponential law of parameter 1. Conditionally on $Z_n =x$ the law of $Z_{n+1}$ is equal in law to $\frac{x}{3}-\frac{1}{3} \log (e_{n+1})$.

As before Figure \ref{fig4} displays the reconstruction of $\lambda$  for a simulated sample with $n=100000$. 

\begin{figure}[h]
   \begin{center}
 \includegraphics[height=7cm]{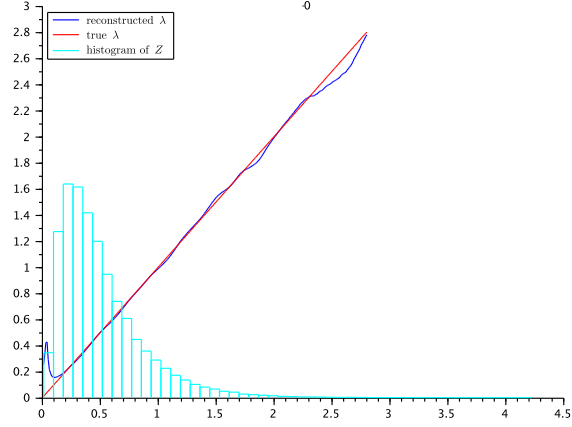}
    \end{center} \caption{\label{fig4}\footnotesize Reconstruction of $\lambda$ for  $n=100000$ in the marked bacteria case.}
\end{figure}

\section{Proof}\label{4}
\subsection{Proof of Proposition \ref{prop transition}}\label{4.1}

We will follow the same idea as in \cite{hof}. We prove a minorisation condition, strong aperiodicity and a drift condition for the transition operator $\pl$ in order to use Theorem 1.1 of \cite {B}.

\begin{proof}[Minorisation condition]
Let $\lambda$ be such that $(\lambda,f, \phi )\in {\mathcal F}(\mathfrak{c},m,M)$ and set $\mathcal{C}=(0,r)$ where $r$ is specified by $\mathfrak{c}$. Fix a measurable $A\in \mathcal{F}$ and $x\in\mathcal{C}$;
thanks to (\ref{gmM}), we have
\[\pl (x,A)\geq\int_{A}\lambda (f^{-1}(y))e^{-\int_{f(x)}^{y}\lambda (f^{-1}(s))M(s)ds}m(y)\nbOne_{\{y\geq f(x)\}}dy. \]
We introduce the function $\varphi_{\lambda}$   \begin{equation}\label{varl}\varphi_{\lambda} (y):=\lambda (f^{-1}(y))e^{-\int_{0}^{y}\lambda (f^{-1}(s))M(s)ds}m(y)\:\; \:\; \:\:\; \:\; \: \forall y\in \R^+ ,\end{equation}  and the measure $\mu_{\lambda}$ \[\mu_{\lambda}(dy):=\frac{\varphi_{\lambda}(y)}{c_{\lambda}}\nbOne_{\{y>f(r)\}}dy,\] where $c_{\lambda}=\int_{f(r)}^{\infty}\vl(u) du$.  Thus, we have
\[\pl (x,A)\geq\mu_{\lambda}(A)c_{\lambda}.\]
By using (\ref{lambdainf}) and (\ref{lambdaL}), we get that
\[c_{\lambda}=\int_{f(r)}^{\infty}\vl(u) du=\left[\frac{-m(y)e^{-\int_{0}^{y}\lambda (f^{-1}(s))M(s)ds}}{M(y)}\right]_{f(r)}^{\infty}\geq \frac{m(f(r))e^{-L}}{M(f(r))}:=\tilde\beta>0.\]
This shows that  the following minorisation condition holds for every $x\in\mathcal{C}$ and $A\in\mathcal{F}$ uniformly in $\lambda$ such that $(\lambda,f, \phi )\in {\mathcal F}(\mathfrak{c},m,M)$:
\begin{equation} \label{minorisation}\pl (x,A)\geq\mu_{\lambda}(A)\tilde\beta. \end{equation}
\end{proof}
\begin{proof}[Strong aperiodicity condition]
We have

\begin{align}
\mu_{\lambda}(\mathcal{C}) \tilde\beta&=c_{\lambda}^{-1}\tilde\beta\int_{f(r)}^{r}\vl (y)dy
=c_{\lambda}^{-1}\tilde\beta\left[\frac{-m(y)e^{-\int_{0}^{y}\lambda (f^{-1}(s))M(s)ds}}{M(y)}\right]_{f(r)}^{r}\nonumber \\
&=\tilde\beta(1-\frac{m(r)M(f(r))}{M(r)m(f(r))}e^{-\int_{f(r)}^{r}\lambda (f^{-1}(s))M(s)ds},)
\end{align}

  using  the computation we just did for $c_{\lambda}$.

Now we use (\ref{lambdal}) to get that 
\begin{equation}\label{aperiodicity}
\mu_{\lambda}(\mathcal{C}) \tilde\beta\geq\tilde\beta\left(1-\frac{m(r)M(f(r))}{M(r)m(f(r))}e^{-l}\right):=\beta>0.\end{equation}

\end{proof}
\begin{proof}[Drift condition] 
Let $\lambda$ be such that $(\lambda,f, \phi )\in {\mathcal F}(\mathfrak{c},m,M)$
and recall that  ${\mathbb V}:{\R^+}\rightarrow [1,\infty)$, which is defined in  \eqref{first def Lyapu}, is  continuously differentiable and satisfies
\begin{equation} \label{control V alinfini}
\lim_{y \rightarrow \infty} {\mathbb V}(y)\exp\big(- \tfrac{a}{b+1}y^{b+1}\big)=0.
\end{equation}
For $x \geq r$, using \eqref{gmM} and   integration by parts with the boundary condition \eqref{lambdainf}, we have, 
\begin{align}
\pl {\mathbb V}(x)  & =  \int_{f(x)}^\infty {\mathbb V}(y)\lambda (f^{-1}(y))e^{-\int_{f(x)}^{y}\lambda (f^{-1}(s))g_{x}(s)ds}g_{x}(y)dy\nonumber \\
& \leq \int_{f(x)}^\infty  {\mathbb V}^{'}(y) e^{-\int_{f(x)}^{y}\lambda (f^{-1}(s))m(s)ds}dy.\nonumber 
\end{align}
Thanks to (\ref{lambdapoly}), we get that

$$\begin{disarray}{rcl}
\pl {\mathbb V}(x)   &  \leq & \int_{f(x)}^\infty  {\mathbb V}^{'}(y) e^{-\int_{f(x)}^{y}as^{b}ds}dy  \leq e^{a\frac{f(x)^{b+1}}{b+1}} \int_{f(x)}^\infty  {\mathbb V}^{'}(y) e^{-a\frac{y^{b+1}}{b+1}} dy.
\end{disarray}$$

Integrating again by parts and using (\ref{control V alinfini}), 
we obtain that

$$\begin{disarray}{rcl}
\pl {\mathbb V}(x)  
   &  \leq &e^{a\frac{f(x)^{b+1}}{b+1}} \int_{f(x)}^\infty  {\mathbb V}(y) ay^{b}e^{-a\frac{y^{b+1}}{b+1}} dy.
\end{disarray}$$

Now use the change of variable $z=a\frac{y^{b+1}}{b+1}$ and  the definition of  ${\mathbb V}(x)$. As \eqref{control V alinfini} is satisfied, we get

$$\begin{disarray}{rcl}
\pl {\mathbb V}(x)  
   &  \leq& {\mathbb V}(x)   \int_{a\frac{f(x)^{b+1}}{b+1}}^\infty  e^{\frac{a}{b+1}\big(f((\frac{z(b+1)}{a})^{1/(b+1)})\big)^{b+1}-z} dz.
\end{disarray}$$

By using (\ref{fk}), we obtain,  for $x\geq r$

$$\begin{disarray}{rcl}
\pl {\mathbb V}(x)  
   &  \leq& {\mathbb V}(x)  \int_{a\frac{f(x)^{b+1}}{b+1}}^\infty  e^{(\kappa^{b+1}-1)z} dz.\end{disarray}$$
 
   Therefore,
\begin{equation} \label{first drift}\pl {\mathbb V}(x)  \leq   {\mathbb V}(x) \delta(\mathfrak{c},f),
\end{equation}
with
\[
\delta(\mathfrak{c},f)= \frac{1}{1-\kappa^{b+1}} \exp\big(- (1-\kappa^{b+1})\tfrac{a}{b+1} (f(r))^{b+1}\big),
\]
and 
we have $\delta(\mathfrak{c},f) < 1$ by Assumption \ref{contrainte constante}.

We next need to control $\pl{\mathbb V}$ outside $x \in [r,\infty)$, that is on the small set ${\mathcal C}$.
For every $x \in {\mathcal C}$, we have
\begin{align}
 \pl{\mathbb V}(x) 
 \leq &\Big(\int_{f(x)}^{f(r)} {\mathbb V}(y) \lambda (f^{-1}(y))g_{x}(y)dy + \int_{f(r)}^\infty {\mathbb V}(y)\lambda (f^{-1}(y))e^{-\int_{f(x)}^{y}\lambda (f^{-1}(s))g_{x}(s)ds}g_{x}(y)dy\Big)\nonumber \\
 \leq &\,  M\sup_{y \in [0,f(r)]}{\mathbb V}(y) L+\delta(\mathfrak{c},f) {\mathbb V}(r) =: K <\infty , \label{second drift} 
\end{align}
where we used \eqref{lambdaL}, \eqref{gmM}, \eqref{first drift} for $x=r$ and the fact that  $(\lambda,f, \phi )\in {\mathcal F}(\mathfrak{c},m,M)$. 
Combining  \eqref{first drift} and \eqref{second drift}, we conclude that
\begin{equation} \label{drift condition}
\pl {\mathbb V}(\boldsymbol{x}) \leq \delta(\mathfrak{c},f) {\mathbb V}(\boldsymbol{x}) \nbOne_{\{\boldsymbol{x} \notin {\mathcal C}\}}+ K\nbOne_{\{\boldsymbol{x}\in {\mathcal C}\}}. 
\end{equation}

\end{proof}
\begin{proof}[End of the proof of Proposition \ref{prop transition}]
 By Theorem 1.1 in Baxendale \cite{B} the minorisation condition \eqref{minorisation} together with the strong aperiodicity condition \eqref{aperiodicity} and the drift condition \eqref{drift condition} imply inequality \eqref{contraction}, with $R$ and $\gamma $ that explicitly depend on $\delta(\mathfrak{c},f)$, $\beta$, $\tilde \beta$, ${\mathbb V}$ and $K$. By construction, this bound is uniform in $\lambda$ such that $(\lambda,f, \phi )\in {\mathcal F}(\mathfrak{c},m,M)$. More specifically, we have
$$\gamma < \min\{\max\{\delta(\mathfrak{c},f), \gamma_{\lambda ,{\mathbb V}}\}, 1\},$$ with $ \gamma_{\lambda ,{\mathbb V}}$ the spectral radius of the operator $\mathcal{P}_\lambda -1\otimes\nu_\lambda$ acting on the  Banach space of functions $\psi :R_+ \rightarrow \R$ such that \[\sup\left\{\frac{|\psi (x)|}{\mathbb{V} (x)},\:\; x\in\R_+\right\}<\infty.\]  
Therefore, under Assumption (\ref{lambdaexis}) we have $\gamma <1$.

It remains to prove equality \eqref{formulel}. As $\pl (x,dy)=\pl (x,y)dy$ and \[\nul\pl =\nul,\] we have that $\nul (dy)=\nul (y)dy$ and
$$\begin{disarray}{rcl}
\nul (y)&=&\int_{\R_+} \nul(x)\pl (x,y)dx\\
&=&\int_{\R_+} \nul(x)\lambda (f^{-1}(y))e^{-\int_{f(x)}^{y}\lambda (f^{-1}(s))g_{x}(s)ds}g_{x}(y)\nbOne_{\{f(x)\leq y\}}dx.
\end{disarray}$$
Thanks to (\ref{lambdainf}), we get that 
\[e^{-\int_{f(x)}^{y}\lambda (f^{-1}(s))g_{x}(s)ds}=\int_{y}^{\infty}\lambda (f^{-1}(s))g_{x}(s) e^{-\int_{f(x)}^{s}\lambda (f^{-1}(s^{'}))g_{x}(s^{'})ds^{'}}ds.\]
Therefore,
\begin{align*}
\nul (y)
=&\lambda (f^{-1}(y))\int_{\R^+} \nul(x)g_{x}(y) \nbOne_{\{f(x)\leq y\}}\int_{y}^{\infty}\lambda (f^{-1}(s))g_{x}(s)e^{-\int_{f(x)}^{s}\lambda (f^{-1}(s^{'}))g_{x}(s^{'})ds^{'}}dsdx
\\=&\lambda (f^{-1}(y))\int_{\R^+} \nul(x)g_{x}(y) \nbOne_{\{f(x)\leq y\}}\int_{y}^{\infty} \nbOne_{\{s\geq y\}}\pl (x,s) dsdx\\=&\lambda (f^{-1}(y))\E_{\nu_\lambda} (g_{Z_0}(y) \nbOne_{\{f(Z_0 )\leq y\}} \nbOne_{\{Z_1 \geq y\}}).
\end{align*}

\end{proof}

\subsection{Rate of convergence for the empirical measure} \label{rate of convergence for the empirical measure}
We now give a few results that we will need for the proof of Theorem \ref{upper bound} in the next Subsection. 
In fact, we decompose the square loss error  into a sum of three terms that we will study in the following Propositions.

The notation $\lesssim$ means inequality up to a constant that not depend on $n$.

\begin{lemma} \label{minoration mes inv}
For any $\mathfrak{c}$ such that Assumptions \ref{contrainte constante} and \ref{hypopdmp} are satisfied, there exists a constant $d(\mathfrak{c}) \geq 0$ such that for any compact interval ${\mathcal D}\subset (d(\mathfrak{c}),\infty)$, we have
$$\inf_{\lambda :\: (\lambda,f, \phi )\in {\mathcal F}(\mathfrak{c},m,M)}\inf_{x \in {\mathcal D}}\varphi_\lambda(x)^{-1}\nul (x) >0,$$
where $\varphi_\lambda (x)$ is defined in \eqref{varl}.
\end{lemma}
\begin{proof} Recall that ${\mathbb V}(x)=\exp\big(\frac{a}{b+1}(f(x))^{b+1}\big)$ for every $x \in [0,\infty)$. By Proposition \ref{prop transition} (and, more precisely, equation \eqref{contraction}) we have
\begin{equation} \label{integ}
\sup_{\lambda: \: (\lambda,f, \phi )\in {\mathcal F}(\mathfrak{c},m,M)}\int_{[0,\infty)}{\mathbb V}(x)\nul(x)dx<\infty,
\end{equation}
 additionally from \eqref{drift condition} in the proof of Proposition \ref{prop transition}, we have that  $\sup_{\lambda: \: (\lambda,f, \phi )\in {\mathcal F}(\mathfrak{c},m,M)}\pl{\mathbb V}(x)<\infty$  for every $x \in \R^+$ . As a consequence, for every $x\in (0,\infty)$, we have
\begin{align*}
\int_{f^{-1}(x)}^\infty \nul (y)dy \;
&\leq \exp\big(-\frac{a}{b+1}(f(x))^{b+1}\big)\int_{[0,\infty)}{\mathbb V}(y)\nul (y)dy,
\end{align*}
and this bound is uniform in $\lambda$ such that  $(\lambda,f, \phi )\in {\mathcal F}(\mathfrak{c},m,M)$ by \eqref{integ}. Therefore, for every $x\in (0,\infty)$, we have
\begin{equation} \label{ineg qui tue}
\sup_{\lambda: (\lambda,f, \phi )\in {\mathcal F}(\mathfrak{c},m,M)} \int_{f^{-1}(x)}^\infty \nul (y)dy \leq c(\mathfrak{c},f)\exp\big(-\frac{a}{b+1}(f(x))^{b+1}\big),
\end{equation}
for some $c(\mathfrak{c},f)>0$. Let 
\begin{equation} \label{def d}
d(\mathfrak{c},f)>f^{-1}\big((\frac{b+1}{a} \log (c(\mathfrak{c},f)))^{1/(b+1)} \big).
\end{equation}
By the definition of $\nul$ and using \eqref{gmM}, for every $y\in (0,\infty)$, we now have
\begin{align*}
\nu_\lambda(y) 
& =  \int_{0}^\infty \nul(x)\lambda (f^{-1}(y))e^{-\int_{f(x)}^{y}\lambda (f^{-1}(s))g_{x}(s)ds}g_{x}(y)\nbOne_{\{f(x)\leq y\}}dx\\& \geq  e^{-\int_{0}^{y}\lambda (f^{-1}(s))M(s)ds}\lambda (f^{-1}(y))m(y)\int_{0}^{f^{-1}(y)} \nul(x)dx\\& \geq  e^{-\int_{0}^{y}\lambda (f^{-1}(s))M(s)ds}\lambda (f^{-1}(y))m(y)\big( 1-\int_{f^{-1}(y)}^{\infty} \nul(x)dx\big)\\& \geq  e^{-\int_{0}^{y}\lambda (f^{-1}(s))M(s)ds}\lambda (f^{-1}(y))m(y)\left( 1-c(\mathfrak{c},f)\exp\left(-\frac{a}{b+1}(f(y))^{b+1}\right)\right)
\end{align*}
where we used \eqref{ineg qui tue} for the last inequality. By \eqref{def d}, for $y \geq d(\mathfrak{c},f)$ we have 
$$\left(1-c(\mathfrak{c},f) \exp\left(-\frac{a}{b+1}(f(y))^{b+1}\right)\right)>0,$$ and the conclusion  follows readily by the  definition of  $\varphi_\lambda$.
\end{proof}

For every $y\in (0,\infty)$, define
\begin{equation} \label{def D}
D(y) = \E_{\nul}\big[g_{Z_{0}}(f(y))\nbOne_{\{Z_{1}\geq f(y),\,f(y) \geq f(Z_{0})\}}\big],
\end{equation}
\begin{equation} \label{def D_n sansw}
D_n(y) = n^{-1}\sum_{k=1}^{n}  g_{Z_{k-1}}(f(y))\nbOne_{\{Z_{k}\geq f(y),\,f(y) \geq f(Z_{k-1})\}},
\end{equation}
and
\begin{equation} \label{def D_n}
D_n(y)_{\varpi_n} = \left(n^{-1}\sum_{k=1}^{n}  g_{Z_{k-1}}(f(y))\nbOne_{\{Z_{k}\geq f(y),\,f(y) \geq f(Z_{k-1})\}}\right)\bigvee \varpi_n.
\end{equation}
\begin{prop} \label{vitessed}
Work under Assumptions \ref{contrainte constante} and  \ref{hypopdmp}. Let $\mu$ be a probability measure  on $\R^+$ such that $\int_{\R^+}{\mathbb V}(\boldsymbol{x})^2\mu(d\boldsymbol{x})<\infty$.
If $1 \geq \varpi_n\rightarrow 0$ as $n \rightarrow \infty$, we have
\begin{equation} \label{convergence D_n}
\sup_{y\in {\mathcal D}}\E_{\mu}\big[\big(D_n(y)_{\varpi_n}-D(y)\big)^2\big] \lesssim n^{-1}
\end{equation}
uniformly in $\lambda$ such that \[ (\lambda,f, \phi )\in {\mathcal F}(\mathfrak{c},m,M),  \:\; \text{and}\:\; \lambda \in  {\mathcal H}^s({f(\mathcal D}), M_{1}),\]
\end{prop}
We first need the following estimate
\begin{lemma} \label{minoration D}
Work under Assumptions \ref{contrainte constante} and \ref{hypopdmp}.
Let $d(\mathfrak{c},f)$ be defined as in Lemma \ref{minoration mes inv}. For every compact interval ${\mathcal D} \subset (d(\mathfrak{c},f),\infty)$ such that $\inf {\mathcal D} \leq f(r)$, we have
$$
\inf_{ \lambda : \: (\lambda,f, \phi )\in {\mathcal F}(\mathfrak{c},m,M) \:\; \text{and}\:\; \lambda \in  {\mathcal H}^s(f({\mathcal D}), M_{1})}\inf_{y\in {\mathcal D}}D(y) > 0.
$$
\end{lemma}
\begin{proof}[Proof ] 
By \eqref{formulel} and the definition of $\varphi_B$ in \eqref{varl}, we readily have that 
$$
D(y) = \frac{\nul (f(y))}{\lambda (y)}= \frac{\nul (f(y))}{\varphi_{\lambda}(f (y))}e^{-\int_{0}^{f(y)}\lambda (f^{-1}(s))M(s)ds}m(y).
$$
Since \[ (\lambda,f, \phi )\in {\mathcal F}(\mathfrak{c},m,M) \:\; \text{and}\:\; \lambda \in {\mathcal H}^s({f(\mathcal D}), M_{1}),\] by applying \eqref{lambdaL}, we obtain
\begin{align*}
\int_{0}^{f(y)}\lambda (f^{-1}(s))M(s)ds&\leq \int_{0}^{\sup {\mathcal D}}\lambda (f^{-1}(s))M(s)ds \\& \leq L+\int_{f(r)}^{\sup \mathcal{D}}\lambda (f^{-1}(s))\sup_{y\in {\mathcal D}}|M(y)|ds 
\\& \leq \left(L + \sup_{y\in {\mathcal D}}|M(y)| M_{1}\sup {\mathcal D}\right) < \infty,
\end{align*} 
where we used that $\inf {\mathcal D} \leq f(r)$. It follows that
$$
\inf_{y\in {\mathcal D}}\exp\left(-\int_{0}^{f(y)}\lambda (f^{-1}(s))M(s)ds\right) \geq \exp\left(- (L + \sup_{y\in {\mathcal D}}|M(y)|M_{1}\sup {\mathcal D}) \right)>0
$$
and Lemma \ref{minoration D} follows by applying Lemma \ref{minoration mes inv}.
\end{proof}
\begin{proof}[Proof of Proposition \ref{vitessed}]
Since $D_n(y)$ is bounded by $M$, we have
\begin{equation} \label{take off varpi}
\big(D_n(y)_{\varpi_n}-D(y)\big)^2\lesssim \big(D_n(y)-D(y)\big)^2+{\bf 1}_{\{D_n(y) < \varpi_n\}}.
\end{equation}
Thus, by Lemma \ref{minoration D}  we may choose $n$ sufficiently large  that 
$$
0 < \varpi_n \leq q=\tfrac{1}{2}\inf_{\lambda : \: (\lambda,f, \phi )\in {\mathcal F}(\mathfrak{c},m,M) \:\; \text{and}\:\; \lambda \in   {\mathcal H}^s(f({\mathcal D}), M_{1})}\inf_{y \in {\mathcal D}}D(y).
$$

Since 
$$\{D_n(y) < \varpi_n\}\subset \{D_n(y)-D(y)< -q\},$$
by integrating \eqref{take off varpi}, we have that $\E_{\mu}\big[\big(D_n(y)_{\varpi_n}-D(y)\big)^2\big]$ is less than a constant times
\begin{align*}
 \E_{\mu}\big[\big(D_n(y)-D(y)\big)^2\big] + \PP_{\mu}\big(|D_n(y)-D(y)| \geq q\big).
\end{align*}
By the Bienaymé-Tchebychev inequality, this quantity is less than a constant times \[\E_{\mu}\big[\big(D_n(y)-D(y)\big)^2\big].\]
Set $G(x,z,y) =g_{x}(f(y))\nbOne_{\{z\geq f(y),\,f(y) \geq f(x)\}}$ and note that $G(\cdot , \cdot , \cdot )$ is bounded on $\R^+$ by $\sup_{y\in {\mathcal D} } |M(y)|$.
It follows that
\begin{align*}
D_n(y) -D(y) & = n^{-1}\sum_{k=1}^{n}\Big( G(Z_{k-1},Z_{k},y)-\E_{\nul}\big[G(Z_{k-1},Z_{k},y)\big]\Big). 
\end{align*}
Therefore, 
\begin{align}\label{d2}
 \E_{\mu}\big[\big(D_n(y)-D(y)\big)^2\big] =\;\frac{1}{n^{2}}\sum_{k,k^{'} \in\{1,..,n\}}
  \E_{\mu}\big[&\big(G(Z_{k-1},Z_{k},y)-\E_{\nul}\big[G(Z_{k-1},Z_{k},y)\big]\big)\nonumber\\&
 \big(G(Z_{k^{'}-1},Z_{k^{'}},y)-\E_{\nul}\big[G(Z_{k^{'}-1},Z_{k^{'}},y)
 \big]\big)\big] .
\end{align}
For $|k-k^{'}|\geq 2$, applying  Markov's property, we get that

 \begin{align*}&\mathbb{E}_{\mu} \big[\big(G(Z_{k-1},Z_{k},y)-\E_{\nul}\big[G(Z_{k-1},Z_{k},y)\big]\big)\\&
 \big(G(Z_{k^{'}-1},Z_{k^{'}},y)-\E_{\nul}\big[G(Z_{k^{'}-1},Z_{k^{'}},y)
 \big]\big)| Z_{i} \forall i\leq  k\wedge k^{'}\big]\\&= \int\int \pl^{k\vee k^{'}- k\wedge k^{'}} (Z_{k\wedge k^{'}},dz)\pl (z,dz^{'})\big(G(z,z^{'},y)-\E_{\nul}\big[G(Z_{0},Z_{1},y)\big]\big)\\&\:\;\:\;\:\;\:\;\:\;\:\;\:\;\big[G(Z_{k\wedge k^{'}-1},Z_{k\wedge k^{'}},y)-\E_{\nul}\big[G(Z_{0},Z_{1},y)
 \big]\big],
\end{align*}
with  $k\wedge k^{'}=\min\{k, k^{'}\}$.

Applying Proposition \ref{prop transition} with $h(z)=\int \pl (z,dz^{'})G(z,z^{'},y)$, we get

$$\begin{disarray}{ll}
&\E_{\mu}\big[\big(G(Z_{k-1},Z_{k},y)-\E_{\nul}\big[G(Z_{k-1},Z_{k},y)\big]\big)
 \big(G(Z_{k^{'}-1},Z_{k^{'}},y)-\E_{\nul}\big[G(Z_{k^{'}-1},Z_{k^{'}},y)
 \big]\big)\big]  \\&
\leq \; R \E_{\mu}\big[{\mathbb V}(Z_{k\wedge k^{'}})\big(G(Z_{k\wedge k^{'}-1},Z_{k\wedge k^{'}},y)-\E_{\nul}\big[G(Z_{0},Z_{1},y)
 \big]\big)\big]\gamma^{k\vee k^{'}- k\wedge k^{'}}
\\&\lesssim \int_{E}\pl^{ k\wedge k^{'}}{\mathbb V}(x)\mu(dx)\, \gamma^{k\vee k^{'}- k\wedge k^{'}},
\end{disarray}$$
as the function $G$ is bounded by $\sup_{y\in {\mathcal D}}|M(y)|$.

For $|k-k^{'}|=1$, we suppose for example that $k^{'}=k-1$.  Applying the Markov property, we get that

\begin{align*}
&\mathbb{E}_{\mu} \big[\big(G(Z_{k-1},Z_{k},y)-\E_{\nul}\big[G(Z_{k-1},Z_{k},y)\big]\big)\\&
 \:\;\:\;\:\;\:\big(G(Z_{k^{'}-1},Z_{k^{'}},y)-\E_{\nul}\big[G(Z_{k^{'}-1},Z_{k^{'}},y)
 \big]\big)| Z_{i}\:\forall i\leq   k-1\big]\\&= \int \pl (Z_{k-1},dz)\big(G(Z_{k-1} ,z,y)-\E_{\nul}\big[G(Z_{k^{'}-1},Z_{k^{'}},y)\big]\big)\\&\:\;\:\;\:\;\:\;\:\;\big[G(Z_{k-2},Z_{k-1},y)-\E_{\nul}\big[G(Z_{0},Z_{1},y)
 \big]\big].
\end{align*}

Applying  Proposition \ref{prop transition} again, we get

$$\begin{disarray}{ll}
&\E_{\mu}\big[\big(G(Z_{k-1},Z_{k},y)-\E_{\nul}\big[G(Z_{k-1},Z_{k},y)\big]\big)
 \big(G(Z_{k^{'}-1},Z_{k^{'}},y)-\E_{\nul}\big[G(Z_{k^{'}-1},Z_{k^{'}},y)
 \big]\big)\big]  \\&
\leq \; R \E_{\mu}\big[{\mathbb V}(Z_{k\wedge k^{'}})\big(G(Z_{k-2},Z_{k-1},y)-\E_{\nul}\big[G(Z_{0},Z_{1},y)
 \big]\big)\big]\gamma^{k\vee k^{'}- k\wedge k^{'}}
\\&\lesssim \int_{E}\pl^{ k\wedge k^{'}}{\mathbb V}(x)\mu(dx)\, \gamma^{k\vee k^{'}- k\wedge k^{'}},
\end{disarray}$$
as the function $G$ is bounded by $\sup_{y\in {\mathcal D}}|M(y)|$.

For $k=k^{'}$,   

\begin{align*}
&\mathbb{E}_{\mu} \big[\big(G(Z_{k-1},Z_{k},y)-\E_{\nul}\big[G(Z_{k-1},Z_{k},y)\big]\big)
 ^2| Z_{i}\:\forall i\leq   k-1\big]\lesssim  \gamma^{k\vee k^{'}- k\wedge k^{'}},
\end{align*}
as the function $G$ is bounded by $\sup_{y\in {\mathcal D}}|M(y)|$.

Moreover as  ${\mathbb V}$ satisfies \eqref{drift condition}, we get
\begin{equation} \label{majoration V deux}
\sup_{\lambda : \: (\lambda,f, \phi )\in {\mathcal F}(\mathfrak{c},m,M)}\pl^{k\wedge k^{'}}{\mathbb V}(x) \lesssim 1+{\mathbb V}(x)
\end{equation}
and, thus, for any $k$ and $k^{'}$,
$$\begin{disarray}{ll}
&\E_{\mu}\big[\big(G(Z_{k-1},Z_{k},y)-\E_{\nul}\big[G(Z_{k-1},Z_{k},y)\big]\big)
 \big(G(Z_{k^{'}-1},Z_{k^{'}},y)-\E_{\nul}\big[G(Z_{k^{'}-1},Z_{k^{'}},y)
 \big]\big)\big]  
\\&\lesssim  \int_{\R_+ }\big(1+{\mathbb V}(x)\big)\mu(dx)\, \gamma^{k\vee k^{'}- k\wedge k^{'}}.
\end{disarray}$$

Since ${\mathbb V}$ is $\mu$-integrable by assumption, thanks to (\ref{d2}), 
we have 
\begin{equation}\label{d22}
 \E_{\mu}\big[\big(D_n(y)-D(y)\big)^2\big] 
 \lesssim\;\frac{1}{n^{2}}\sum_{k,k^{'} \in\{1,2,..,n\}} \gamma^{k\vee k^{'}- k\wedge k^{'}}\lesssim n^{-1},
\end{equation}
uniformly in $y\in {\mathcal D}$ and $\lambda$ such that $ (\lambda,f, \phi )\in {\mathcal F}(\mathfrak{c},m,M)$. 

\end{proof}
\begin{prop}\label{moment lemma bis}
Work under Assumptions \ref{contrainte constante},   \ref{prop K} and  \ref{hypopdmp}. Let $\mu$ be a probability measure  on $\R^+$ such that $\int_{\R^+}{\mathbb V}(\boldsymbol{x})^2\mu(d\boldsymbol{x})<\infty$.
Then we have
\begin{equation} \label{moment mes inv}
\sup_{y \in {\mathcal D}}\E_{\mu}\big[\big(K_{h_n}\star \widehat \nu_n(y)-K_{h_n} \star \nul(y)\big)^2\big] \lesssim {(nh_n)}^{-1}
\end{equation}
uniformly in $\lambda$ such that $(\lambda,f, \phi )\in {\mathcal F}(\mathfrak{c},m,M)$ and $\|f\|_{L^\infty({\mathcal D})} \leq M_2$ with 
$\widehat \nu_n ( \cdot ) =\frac{1}{n}\sum_{k\in\{1,...,n\}}\nbOne_{\{ Z_k\}} ( \cdot  ) $
\end{prop}
\begin{proof}
By definition,
\begin{align}
 \E_{\mu}\big[\big(K_{h_n}\star \widehat \nu_n(y)-K_{h_n} \star \nul(y)\big)^2\big] 
& =\, (nh_n)^{-2}\E_{\mu}\Big[\Big(\sum_{k\in\{1,...,n\}}K\big(\tfrac{Z_{k}-y}{h_n}\big)-\E_{\nu_\lambda }\big[K\big(\tfrac{Z_{0}-y}{h_n}\big)\big]\Big)^2\Big] \nonumber\\
 &=\, (nh_n)^{-2}\sum_{k,k^{'}\in\{1,...,n\}}\mathbb{E}_{\mu}\big[\widetilde K\big(\tfrac{Z_{k}-y}{h_n}\big) \widetilde K\big(\tfrac{Z_{k^{'}}-y}{h_n}\big)\big], \nonumber
\end{align}
with $\widetilde K(\tfrac{Z_{k}-y}{h_n}) =K\big(\tfrac{Z_{k}-y}{h_n}\big)-\E_{\nul}\big[K\big(\tfrac{Z_{0}-y}{h_n}\big)\big]$ .

As in the proof of Proposition \ref{vitessed},  thanks to the Markov property we obtain
\begin{align}
&\mathbb{E}_{\mu}\big[\mathbb{E}_{\mu}(\widetilde K\big(\tfrac{Z_{k}-y}{h_n}\big) \widetilde K\big(\tfrac{Z_{k^{'}}-y}{h_n}\big)\big| Z_{k\wedge k^{'}}\big] \nonumber\\
&= \; \big[\pl^{k\vee k^{'}- k\wedge k^{'}} J(Z_{k\wedge k^{'}})-\E_{\nul} (K\big(\tfrac{Z_{0}-y}{h_n}\big))\big]\big(  K(\tfrac{Z_{k\wedge k^{'}}-y}{h_n})-\E_{\nul} (K\big(\tfrac{Z_{0}-y}{h_n}\big))\big)\nonumber \label{decomp markov}
,
\end{align}
with $J(\cdot )=K(\frac{\cdot-y}{h_n})$.
First, as $K$ has bounded support, $K \lesssim {\mathbb V}$ and so we  can apply \eqref{contraction} from Proposition \ref{prop transition}. We obtain
\begin{equation} \label{premiere maj h}
\big|\big[\pl^{k\vee k^{'}- k\wedge k^{'}} J(Z_{k\wedge k^{'}})-\E_{\nul} (K\big(\tfrac{Z_{0}-y}{h_n}\big))\big]
\big| \leq R{\mathbb V}\big(Z_{k\wedge k^{'}}\big)\gamma^{k\vee k^{'}- k\wedge k^{'}}.
\end{equation}

Moreover, thanks to (\ref{formulel}) and (\ref{gmM}) and the fact that $\lambda\in\mathcal{H}^{s} (f(\mathcal{D}), M_{1})$, we have that 
\begin{equation}
 \sup_{x\in f(\mathcal{D})}\nul (x)\leq M_{1}M,\label{supmu}
\end{equation} 
so that
\begin{align}
\big|\E_{\nul} (K\big(\tfrac{Z_{0}-y}{h_n}\big)) \big|& \leq \int_{[0,\infty)}\big| K\big(\tfrac{x-y}{h_{n}}\big)\big|\nul (x)dx \lesssim h_{n}.
\label{esperance h}
\end{align}

Putting together \eqref{premiere maj h} and \eqref{esperance h}
we derive

\begin{align*}
&
\mathbb{E}_{\mu}\big[\mathbb{E}_{\mu}(\widetilde K\big(\tfrac{Z_{k}-y}{h_n}\big) \widetilde K\big(\tfrac{Z_{k^{'}}-y}{h_n}\big)\big| Z_{k\wedge k^{'}}\big)\big]
 \\&\lesssim \mathbb{E}_{\mu}(RJ(Z_{k\wedge k^{'}}){\mathbb V}\big(Z_{k\wedge k^{'}}\big))\gamma^{k\vee k^{'}- k\wedge k^{'}}
 +\mathbb{E}_{\mu}(R{\mathbb V}\big(Z_{k\wedge k^{'}}\big))\gamma^{k\vee k^{'}- k\wedge k^{'}}h_{n}.
\end{align*}

On the one hand, by using  the Markov property, the fact that  \[ e^{-\int_{f(Z_{k\wedge k^{'}-1})}^{u}\lambda (f^{-1}(s))g_{Z_{k\wedge k^{'}-1}}(s)ds}\leq 1,\] (\ref{gmM}),  that $\lambda\in\mathcal{H}^{s} (f(\mathcal{D}), M_{1})$ and as $f$ is increasing, we can bound ${\mathbb V}(u)\nbOne_{\{u\in\mathcal{D}\}}$ by $e^{\frac{a}{b+1} M_2^{b+1}}$, we get that
\begin{align*}
&\mathbb{E}_{\mu}\big[\big| J\big(Z_{k\wedge k^{'}}\big){\mathbb V}\big(Z_{k\wedge k^{'}}\big)\big| \big]\nonumber \\& \leq \mathbb{E}_{\mu} \Big|\int_{f(Z_{k\wedge k^{'}-1})}^\infty K\big(\tfrac{u-y}{h_n}\big)\lambda (f^{-1}(u))e^{-\int_{f(Z_{k\wedge k^{'}-1})}^{u}\lambda (f^{-1}(s))g_{Z_{k\wedge k^{'}-1}}(s)ds}g_{Z_{k\wedge k^{'}-1}}(u){\mathbb V}(u)du\Big| \nonumber  \\
& \leq \left(\sup_{x\in\mathcal{D}}|M(x)|\right) M_{1}\int_{[0,\infty)}\mathbb{E}_{\mu} \big|K\big(h_{n}^{-1}(u-y)\big)\big| du    \lesssim h_{n} \label{obtention h}
\end{align*}
as    $K$ has compact support.
On the other hand, as ${\mathbb V}^2$ and  ${\mathbb V}$ are $\mu$ integrable by assumption we get that

\begin{align*}
&
\mathbb{E}_{\mu}\big[\mathbb{E}_{\mu}(\widetilde K\big(\tfrac{Z_{k}-y}{h_n}\big) \widetilde K\big(\tfrac{Z_{k^{'}}-y}{h_n}\big)\big| Z_{k\wedge k^{'}}\big)\big]
\lesssim h_{n} \gamma^{k\vee k^{'}- k\wedge k^{'}}.
\end{align*}

Therefore,
\begin{align*}
& \E_{\mu}\big[\big(K_{h_n}\star \widehat \nu_n(y)-K_{h_n} \star \nul(y)\big)^2\big] 
\lesssim\, (nh_n)^{-2}\sum_{k,k^{'} \in\{1,...,n\}}h_{n} \gamma^{k\vee k^{'}- k\wedge k^{'}}
\lesssim\, (nh_n)^{-1}.
\end{align*}

\end{proof}

\subsection{ Proof of Theorem \ref{upper bound}}\label{4.3}

Recall that
$$\widehat \lambda_n (y) = \frac{n^{-1}\sum_{k=1}^{n}  K_{h_{n}}(Z_{k}-f(y))}{n^{-1}\sum_{k=1}^{n}  g_{Z_{k-1}}(f(y))\nbOne_{\{Z_{k}\geq f(y),\,f(y) \geq f(Z_{k-1})\}}\bigvee \varpi_n}$$
and
$$\lambda (y)=\frac{\nul (f(y))}{\E_{\nul}(g_{Z_0}(f(y)) \nbOne_{\{f(Z_0 )\leq f(y)\}} \nbOne_{\{Z_1 \geq f(y)\}})}.$$
We will use the  decomposition
$$\widehat \lambda_n(y)-\lambda (y) = (I+II+III),$$
where
\begin{align*}
I & =\frac{K_{h_n}\star \nu_\lambda (f(y))-\nu_\lambda (f(y))}{D(y)}, \\
II & =\frac{K_{h_n}\star \widehat \nu_n(f(y))-K_{h_n}\star \nu_\lambda (f(y))}{D_n(y)_{\varpi_n}},\\
III & =\frac{K_{h_n}\star \nu_\lambda (f(y))}{D_n(y)_{\varpi_n} D(y)}\big(D(y)-D_n(y)_{\varpi_n}\big), 
\end{align*}
and where $D(y)$ and $D_n(y)_{\varpi_n}$ are defined in \eqref{def D} and \eqref{def D_n} respectively.
It follows that
\begin{align*}
\|\widehat \lambda_n-\lambda\|_{L^2({\mathcal D})}^2  & = \int_{{\mathcal D}}\big(\widehat \lambda_n(y)-\lambda (y)\big)^2dy 
 \lesssim  IV+V+VI,
\end{align*}
where
\begin{align*}
IV & = \int_{{\mathcal D}}\big(K_{h_n}\star \nu_\lambda (f(y))-\nu_\lambda (f(y))\big)^2\tfrac{1}{D(y)^2}dy\\
V & =  \int_{{\mathcal D}}\big(K_{h_n}\star \widehat \nu_n(f(y))-K_{h_n}\star \nu_\lambda (f(y))\big)^2D_n(y)_{\varpi_n}^{-2}dy\\
VI & = \int_{{\mathcal D}} \big(D_n(y)_\varpi-D(y)\big)^2\big(K_{h_n}\star \nu_\lambda (f(y))\big)^2\big(D_n(y)_{\varpi_n} D(y)\big)^{-2}dy.
\end{align*}
\begin{proof}[The term IV] We get rid of the term $\tfrac{1}{D(y)^2}$ using Lemma \ref{minoration D}. By Assumption \ref{prop K} and classical kernel approximation, we have for every $0 < s \leq n_0$
\begin{equation} \label{biais mesure invariante}
IV \lesssim \|K_{h_n}\star \nu_\lambda -\nu_\lambda \big\|_{L^2(f({\mathcal D}))}^2 \lesssim |\nu_\lambda |_{{\mathcal H}^s(f({\mathcal D}))}^2h_n^{2s}.
\end{equation}
\begin{lemma} \label{reg transfer} We work under  Assumption \ref{hypopdmp}. Let ${\mathcal D}\subset (0,\infty)$ be a compact interval, $\lambda \in {\mathcal F}(\mathfrak{c}) $ and $(\lambda,f, \phi )\in {\mathcal F}(\mathfrak{c},m,M)$ for some ${\mathfrak{c}}$ satisfying Assumption \ref{contrainte constante}, $g_x \in {\mathcal H}^s({\mathcal D})$ and $f^{-1} \in{\mathcal H}^s({\mathcal D})$. Then we have
$$\|\nu_\lambda \|_{{\mathcal H}^s(\mathcal{D})} \leq \psi\big( {\mathcal D},\|\lambda\|_{{\mathcal H}^s({\mathcal D})},\|g_x \|_{{\mathcal H}^s({\mathcal D})},\|f^{-1} \|_{{\mathcal H}^s({\mathcal D})}\big)$$
for some continuous function $\psi$. 
\end{lemma}
\begin{proof}[Proof of Lemma \ref{reg transfer}] 
We first recall that
\[\nul (y)=\lambda (f^{-1}(y))\int_{E} \nul(x)e^{-\int_{f(x)}^{y}\lambda (f^{-1}(s))g_{x}(s)ds}g_{x}(y)\nbOne_{\{f(x)\leq y\}}dx.\]
Define 
$$\Lambda_\lambda(x,y) = e^{-\int_{f(x)}^{y}\lambda (f^{-1}(s))g_{x}(s)ds}g_{x}(y).$$
If $\lambda\in{\mathcal F}(\mathfrak{c})$, then $\Lambda_\lambda (x,.) \in {\mathcal H}^s({\mathcal D})$ for every $y \in [0,\infty)$, and we have
$$\|\Lambda_\lambda(x,.)\|_{{\mathcal H}^s({\mathcal D})} \leq \psi_1(\|\lambda\|_{{\mathcal H}^s({\mathcal D})},\|g_x \|_{{\mathcal H}^s({\mathcal D})},\|f^{-1} \|_{{\mathcal H}^s({\mathcal D})} )$$
for some continuous function $\psi_1$. The result is then a consequence of the representation
$\nul (y) = \lambda (f^{-1}(y))\int_0^{f^{-1} (y)} \Lambda_\lambda(x,y)dx$.
\end{proof}
Returning to \eqref{biais mesure invariante} we deduce from Lemma \ref{reg transfer} that $\|\nu_\lambda \|_{{\mathcal H}^s(f(\mathcal{D}))}$ is bounded above by
a constant that depends on  ${\mathcal D}$,$\|g_x \|_{{\mathcal H}^s(f({\mathcal D}))}$, $\|f^{-1} \|_{{\mathcal H}^s(f({\mathcal D}))}$ and $\|\lambda\|_{{\mathcal H}^s(f({\mathcal D}))}$ only. It follows that 
\begin{equation} \label{controle IV}
IV \lesssim h_n^{2s}
\end{equation}
uniformly in $\lambda \in {\mathcal H}^s({\mathcal D},M_1 )$.
\end{proof}
\begin{proof}[The term V] We have
$$\E_{\mu}[V] \leq 
\varpi_n^{-2}|{\mathcal D}| \sup_{y \in {\mathcal D}} \E_{\mu}\big[\big(K_{h_n}\star \widehat \nu_n(f(y))-K_{h_n} \star \nu_\lambda (f(y))\big)^2\big].$$
By \eqref{moment mes inv} of Proposition \ref{moment lemma bis} we derive that
\begin{equation} \label{controle V}
\E_{\mu}[V] \lesssim \varpi_n^{-2}(nh_n)^{-1}
\end{equation}
uniformly in $\lambda  \in {\mathcal F}(\mathfrak{c})$
\end{proof}
\begin{proof}[The term VI] First, thanks to  Lemma \ref{minoration D}, we get that
$$\inf_{ \lambda : \: (\lambda,f, \phi )\in {\mathcal F}(\mathfrak{c},m,M) \:\; \text{and}\:\; \lambda \in  {\mathcal H}^s(f({\mathcal D}), M_{1})}\inf_{y\in f({\mathcal D})}D_n(y)_\varpi D(y) \gtrsim \varpi_n .$$
 Next, 
\begin{align}
\sup_{y \in  f({\mathcal D})} |K_{h_n}\star \nu_\lambda (y)| & = \sup_{y \in  f({\mathcal D})}\big|\int_{[0,\infty)}K_{h_n}(z-y)\nu_\lambda (z)dz\big|  \leq \sup_{y \in {\mathcal D}_1}\nu_\lambda (y)\|K\|_{L^1([0,\infty))}, \label{borne phi l1}
\end{align}
where ${\mathcal D}_1 = \{y+z,\;y\in  f({\mathcal D}),\;z\in \text{supp}(K_{h_n})\} \subset \widetilde {\mathcal D}$, for some compact interval $\widetilde {\mathcal D}$ since $K$ has compact support by Assumption \ref{prop K}. Thanks to (\ref{supmu}), we see that \eqref{borne phi l1} holds uniformly in $\lambda $ such that  $(\lambda,f, \phi )\in {\mathcal F}(\mathfrak{c},m,M)$. We derive that
$$\E_{\mu}\big[VI\big] \lesssim \varpi_n^{-2}\sup_{y \in f({\mathcal D})}\E_{\mu}\big[\big(D_n(y)_{\varpi_n}-D(y)\big)^2\big].$$
Applying \eqref{convergence D_n} of Proposition \ref{vitessed}, we conclude that
\begin{equation} \label{controle VI}
\E_{\mu}\big[VI\big] \lesssim \varpi_n^{-2}n^{-1},
\end{equation}
uniformly in $\lambda $ such that $ (\lambda,f, \phi )\in {\mathcal F}(\mathfrak{c},m,M)$.
\end{proof}
\begin{proof}[End of the proof of Theorem \ref{upper bound}] 
We put together the three estimates \eqref{controle IV}, \eqref{controle V} and \eqref{controle VI} to obtain
\begin{align*}
\E_{\mu}\big[\|\widehat \lambda_n-\lambda\|_{L^2({\mathcal D})}^2\big] \lesssim  h_n^{2s}+\varpi_n^{-2}(nh_n)^{-1}+ \varpi_n^{-2}n^{-1}
\end{align*}
uniformly in $\lambda \in{\mathcal H}^s({\mathcal D},M_1)$ and $ (\lambda,f, \phi )\in {\mathcal F}(\mathfrak{c},m,M)$. The choice $h_n \sim n^{-1/(2s+1)}$  yields the rate $\varpi_n^{-2} n^{-2s/(2s+1)}$.
\end{proof}

\section*{Acknowledgements}
The research of  N. Krell  is partly supported by the Agence Nationale de la Recherche
PIECE 12-JS01-0006-01.

\end{document}